\documentclass[reqno, 12pt]{amsart}
\def\ssign{\textsection\nobreak\hspace{1pt plus 0.3pt}}
\makeatletter
\let\origsection=\section 
\def\mysection{\@mystartsection{section}{1}\z@{.7\linespacing\@plus\linespacing}{.5\linespacing}{\normalfont\scshape\centering\ssign}}
\def\section{\@ifstar{\origsection*}{\mysection}}
\def\appendix{\par\c@section\z@ \c@subsection\z@
   \let\sectionname\appendixname
   \let\section=\origsection
   \def\thesection{\@Alph\c@section}} 
\def\@mystartsection#1#2#3#4#5#6{\if@noskipsec \leavevmode \fi
 \par \@tempskipa #4\relax
 \@afterindenttrue
 \ifdim \@tempskipa <\z@ \@tempskipa -\@tempskipa \@afterindentfalse\fi
 \if@nobreak \everypar{}\else
     \addpenalty\@secpenalty\addvspace\@tempskipa\fi
 \@dblarg{\@mysect{#1}{#2}{#3}{#4}{#5}{#6}}}
\def\@mysect#1#2#3#4#5#6[#7]#8{\edef\@toclevel{\ifnum#2=\@m 0\else\number#2\fi}\ifnum #2>\c@secnumdepth \let\@secnumber\@empty
  \else \@xp\let\@xp\@secnumber\csname the#1\endcsname\fi
  \@tempskipa #5\relax
  \ifnum #2>\c@secnumdepth
    \let\@svsec\@empty
  \else
    \refstepcounter{#1}\edef\@secnumpunct{\ifdim\@tempskipa>\z@ \@ifnotempty{#8}{\@nx\enspace}\else
        \@ifempty{#8}{.}{\@nx\enspace}\fi
    }\@ifempty{#8}{\ifnum #2=\tw@ \def\@secnumfont{\bfseries}\fi}{}\protected@edef\@svsec{\ifnum#2<\@m
        \@ifundefined{#1name}{}{\ignorespaces\csname #1name\endcsname\space
        }\fi
      \@seccntformat{#1}}\fi
  \ifdim \@tempskipa>\z@ \begingroup #6\relax
    \@hangfrom{\hskip #3\relax\@svsec}{\interlinepenalty\@M #8\par}\endgroup
    \ifnum#2>\@m \else \@tocwrite{#1}{#8}\fi
  \else
  \def\@svsechd{#6\hskip #3\@svsec
    \@ifnotempty{#8}{\ignorespaces#8\unskip
       \@addpunct.}\ifnum#2>\@m \else \@tocwrite{#1}{#8}\fi
  }\fi
  \global\@nobreaktrue
  \@xsect{#5}}
\makeatother

\usepackage{amsmath,amssymb,amsthm}
\usepackage{mathrsfs}
\usepackage{mathabx}\changenotsign
\usepackage{dsfont}

\usepackage{graphicx}
\usepackage{float}

\usepackage[dvipsnames]{xcolor}
\usepackage{hyperref}
\hypersetup{
	colorlinks,
	linkcolor={red!60!black},
	citecolor={green!60!black},
	urlcolor={blue!60!black}
}

\definecolor{codelightgray}{gray}{0.8}
\definecolor{codeverylightgray}{gray}{0.9}

\usepackage{array,multirow,colortbl,enumerate}

\usepackage{graphicx}

\usepackage[open,openlevel=2,atend]{bookmark}

\usepackage[abbrev,msc-links,backrefs]{amsrefs}
\usepackage{amsrefs}
\usepackage{doi}

\renewcommand{\PrintDOI}[1]{\doi{#1}}

\usepackage[OT2, T1]{fontenc}
\usepackage{lmodern}
\usepackage[babel]{microtype}
\usepackage[english]{babel}

\DeclareRobustCommand{\rn}[1]{  {\fontencoding{OT2}\selectfont#1}}

\usepackage{geometry}
\numberwithin{equation}{section}
\numberwithin{figure}{section}

\usepackage{enumitem}
\def\rmlabel{\upshape({\itshape \roman*\,})}

\def\alabel{\upshape({\itshape \alph*\,})}

\def\nlabel{\upshape({\itshape \arabic*\,})}

\let\polishlcross=\l
\def\l{\ifmmode\ell\else\polishlcross\fi}

\def\paragraph#1{	\noindent\textbf{#1.}\enspace}

\let\sm=\setminus

\makeatletter
\def\moverlay{\mathpalette\mov@rlay}
\def\mov@rlay#1#2{\leavevmode\vtop{   \baselineskip\z@skip \lineskiplimit-\maxdimen
		\ialign{\hfil$\m@th#1##$\hfil\cr#2\crcr}}}
\newcommand{\charfusion}[3][\mathord]{
	#1{\ifx#1\mathop\vphantom{#2}\fi
		\mathpalette\mov@rlay{#2\cr#3}
	}
	\ifx#1\mathop\expandafter\displaylimits\fi}
\makeatother

\newcommand{\dcup}{\charfusion[\mathbin]{\cup}{\cdot}}

\DeclareFontFamily{U}  {MnSymbolC}{}
\DeclareSymbolFont{MnSyC}         {U}  {MnSymbolC}{m}{n}
\DeclareFontShape{U}{MnSymbolC}{m}{n}{
	<-6>  MnSymbolC5
	<6-7>  MnSymbolC6
	<7-8>  MnSymbolC7
	<8-9>  MnSymbolC8
	<9-10> MnSymbolC9
	<10-12> MnSymbolC10
	<12->   MnSymbolC12}{}
\DeclareMathSymbol{\powerset}{\mathord}{MnSyC}{180}

\usepackage{tikz}
\usetikzlibrary{calc,decorations.pathmorphing,decorations.pathreplacing}
\usetikzlibrary{intersections}
\usetikzlibrary {arrows.meta} 
\pgfdeclarelayer{background}
\pgfdeclarelayer{foreground}
\pgfdeclarelayer{front}
\pgfsetlayers{background,main,foreground,front}

\usepackage{subcaption}

\let\epsilon=\varepsilon

\let\rho=\varrho
\let\theta=\vartheta

\def\FF{{\mathds F}}

\def\ZZ{{\mathds Z}}

\theoremstyle{plain}
\newtheorem{thm}{Theorem}[section]

\newtheorem{prop}[thm]{Proposition}

\newtheorem{clm}[thm]{Claim}
\newtheorem{fact}[thm]{Fact}
\newtheorem{cor}[thm]{Corollary}
\newtheorem{lem}[thm]{Lemma}
\newtheorem{conclusion}[thm]{Conclusion}

\theoremstyle{definition}

\newtheorem{dfn}[thm]{Definition}

\usepackage{accents}

\let\lra=\longrightarrow
\let\phi=\varphi

\DeclareSymbolFont{stmry}{U}{stmry}{m}{n}
\DeclareMathSymbol\arrownot\mathrel{stmry}{"58}
\DeclareMathSymbol\Arrownot\mathrel{stmry}{"59}

\def\Sym{\mathrm{Sym}}
\def\aff{\mathrm{aff}}

\let\vn=\varnothing

\begin{document}
\title[Lev's periodicity conjecture]{On Lev's periodicity 
conjecture}
\author[Christian Reiher]{Christian Reiher}
\address{Fachbereich Mathematik, Universit\"at Hamburg, Hamburg, Germany}
\email{christian.reiher@uni-hamburg.de }

\begin{abstract}
	We classify the sum-free subsets of $\FF_3^n$ whose density 
	exceeds~$\frac16$. This yields a resolution of Vsevolod Lev's 
	periodicity conjecture, which asserts that if a  
	sum-free subset~${A\subseteq \FF_3^n}$ is maximal with respect 
	to inclusion and aperiodic (in the sense that there 
	is no non-zero vector~$v$ satisfying $A+v=A$), 
	then $|A|\le \frac12(3^{n-1}+1)$ ---a bound known to be optimal 
	if $n\ne 2$, while for $n=2$ there are no such sets. 
\end{abstract}
	
\maketitle
	
\section{Introduction}

A subset $A$ of an abelian group $G$ is said to be {\it sum-free}
if the equation $x+y=z$ has no solutions with $x, y, z\in A$.
Thus Fermat's last theorem is equivalent to the statement that for 
every $k\ge 3$ the set of positive $k^{\mathrm{th}}$ powers is 
sum-free; this motivated Schur~\cite{Schur} to prove his 
famous theorem that the set of positive integers cannot be represented 
as a union of finitely many sum-free sets. Since the 1960's the 
study of sum-free sets has intensified due to the problems and 
results of Erd\H{o}s (see e.g.~\cite{E65}). We refer 
to~\cites{Green, EGM} for some recent milestones 
in the area and to~\cite{TV17} for a survey.

When the ambient group is a binary vector space $G=\FF_2^n$, sum-free 
sets have connections to projective geometry and coding theory. 
It is easily seen that in this case the sum-free sets of largest 
cardinality are hyperplanes not passing through the origin. Davydov 
and Tombak~\cite{DT} obtained the much deeper result that 
sum-free sets which cannot be covered by a hyperplane have size at 
most $5\cdot 2^{n-4}$. They also have a stronger theorem on 
{\it maximal sum-free sets} in~$\FF_2^n$, i.e., sum-free sets
that are maximal with respect to inclusion. Notably they showed 
that every such set $A$ possessing more than $2^{n-2}+1$ elements 
is {\it periodic}, which means that there is a non-zero vector~$v$
such that $A+v=A$.

Vsevolod Lev~\cites{VL05, VL23} initiated the study of similar 
questions in other finite vector spaces. Following the title 
of~\cite{VL05} we call vector spaces over the three-element 
field $\FF_3$ {\it ternary}. By an early observation 
of Wallis, Street, and Wallis~\cite{WSW}*{Corollary~7.11} 
sum-free subsets of~$\FF_3^n$ have at most the 
size $3^{n-1}$, the extremal examples again being  
hyperplanes not passing through the origin. The main result 
of~\cite{VL05} asserts that if a sum-free subset of $\FF_3^n$
has more than~${5\cdot 3^{n-3}}$ elements, then it can be covered 
by a hyperplane; this bound is indeed optimal whenever~${n\ge 3}$. 
Furthermore, Lev conjectured in~\cite{VL05} that the largest possible 
cardinality of an aperiodic sum-free subset of~$\FF_3^n$ 
is $\frac12(3^{n-1}+1)$ (when $n\ne 2$) and provided a construction 
exemplifying the lower bound. Despite being included as Problem~7.10
into the well-known problem collection of Croot and Lev~\cite{CL07},
the conjecture has remained open for almost two decades. 

Before we formulate our main result, which classifies all sum-free 
subsets of~$\FF_3^n$ whose cardinality is at 
least $\frac12(3^{n-1}+1)$, we would like to fix some notation 
concerning affine subspaces. Given an affine subspace $U$ of 
a vector space $V$ we write $[U]=U-U$ for the translate of $U$ 
passing through the origin, and $\dim(U)=\dim([U])$ for its dimension. 
When we have two affine subspaces $U\subseteq H\subseteq V$, then 
$H/U=H/[U]$ means the set of translates of $U$ comprising $H$ 
and $\dim(H/U)$ is shorthand for $\dim(H)-\dim(U)$, i.e., the 
dimension of the affine subspace~$H/[U]$ of the quotient 
space~$V/[U]$. As usual, $H$ is called a {\it hyperplane} 
if $\dim(V/H)=1$. The affine space generated by an arbitrary 
set $X\subseteq V$ is designated by~$\aff(X)$. 
Finally, if~$U$ is an affine subspace of~$V$, 
then the vector subspace of~$V$ generated by $U$ is called 
the {\it cone} over $U$ and denoted by~$C(U)$. For instance, 
if~$V$ is ternary and $0\not\in U$, 
then $C(U)$ is the disjoint union of $[U]$, $U$, and $-U$.      
 
\begin{dfn}\label{dfn:halves}
	Given two affine subspaces $U\subseteq H$ of a ternary vector 
	space we call a set $W\subseteq H$ an {\it $(H, U)$-half} 
	if $W+[U]=W$ and $H$ is the disjoint union of $U$, $W$, 
	and $(-U)+(-W)$.    		
\end{dfn}

Notice that the condition $W+[U]=W$ just means that $W$ is a 
union of translates of~$U$. Moreover, $H/U$ 
consists of $U$ itself and $\frac12(|H/U|-1)$ 
pairs $\{U+x, U-x\}$; the second demand on~$W$ is equivalent to~$W$
containing exactly one translate from each such pair. For this 
reason, all $(H, U)$-halves $W$ have the same 
cardinality $|W|=\frac12(|H|-|U|)$. We are now ready for the central 
definition of this article, which proceeds by recursion on the 
dimension of the ambient ternary space $V$.   

\begin{dfn}\label{dfn:prim}
	Given a finite ternary space $V$ we call a subset $A\subseteq V$ 
	{\it primitive} if 
		\begin{enumerate}[label=\alabel]
		\item\label{it:pra} either $A$ is a hyperplane not containing 
			the origin,
		\item\label{it:prb} or there are a hyperplane 
			$H\subseteq V$ not passing through the origin, 
			a proper affine subspace~${U\subsetneqq H}$, 
			an $(H, U)$-half~$W$, and a primitive set $X\subseteq C(U)$ 
			such that 
			\begin{enumerate}[label=\rmlabel]
				\item\label{it:pri} $A=W\cup X$,
				\item\label{it:prii} $X\cap [U]=\vn$,
				\item\label{it:priii} $\dim(H/U)\ge 2$ or $X\ne -U$,
				\item\label{it:priv} and $\aff(X\cap (-U))=-U$.
			\end{enumerate}
	\end{enumerate}
\end{dfn}

As it turns out, primitive sets are the same as dense maximal 
sum-free sets. Their definition might nevertheless appear complicated 
on first reading, and thus we would like to discuss a few 
low-dimensional examples before moving on. In $\FF_3$ itself the only 
primitive sets are $\{1\}$ and $\{2\}$. Assume next that there
existed a primitive subset of $\FF_3^2$ falling under 
clause~\ref{it:prb}. Without loss of generality~$H$ is the 
line $\{(1, y)\colon y\in \FF_3\}$ and $U=\{(1, 0)\}$.
Now $X$ has to be a primitive subset of $C(U)=\FF_3\times\{0\}$,
whence either $X=\{(-1, 0)\}$ or $X=\{(1, 0)\}$. But these two 
cases are ruled out by~\ref{it:priii} and~\ref{it:priv}, respectively. 
Altogether, the primitive subsets of~$\FF_3^2$ are exactly the lines 
avoiding the origin. 

Next we look at $V=\FF_3^3$ and $H=\{1\}\times \FF_3^2$, say. 
The case $\dim(U)=1$ is easily dismissed as in the previous 
paragraph, but when $U$ is a single point we can take $X=-U$
without violating~\ref{it:priii}. This yields primitive sets 
such as 
\[
	P=\{(1, 0, 1), (1, 1, -1), (1, 1, 0), (1, 1, 1), (-1, 0, 0)\}\,.
\]
As the definition of $(H, U)$-halves involves four binary choices
when $\dim(H/U)=2$, one might think at first that there also are 
some `essentially different' primitive sets of this type. 
It could be shown
by means of a simple case analysis, however, that all five-element 
primitive subsets of $\FF_3^3$ are `isomorphic' to $P$, i.e., that 
they are in the orbit of $P$ with respect to the automorphism group 
of $\FF_3^3$. 

The configuration $P$ was first found by Lev~\cite{VL05}, who
observed that for $n\ge 3$ the maximal sum-free 
set $P\times \FF_3^{n-3}$ cannot be covered by a hyperplane, 
which shows that the bound~$5\cdot 3^{n-3}$ in his aforementioned 
result is optimal. He further observed that for $n\ge 3$, $\dim(U)=0$, 
and~${X=-U}$ Definition~\ref{dfn:prim}\ref{it:prb} yields an 
aperiodic maximal sum-free subset of $\FF_3^n$ whose size
is~$\frac12(3^{n-1}+1)$.

In general, it turns out that the cardinality of a primitive set only 
depends on the dimension of its period space. Before making this 
precise we would like to recall that for a nonempty subset $A$
of an abelian group $G$ its {\it symmetry group} is defined by 
\[
	\Sym(A)=\{x\in G\colon A+x=A\}\,.
\]
If~$G$ is a ternary vector space, then $\Sym(A)$ is a linear subspace 
of~$G$, and $A$ is aperiodic if and only if $\Sym(A)=\{0\}$. As we 
shall see in the next section, a simple induction along the recursive
formation of primitive sets reveals the following.  
 
\begin{lem}\label{lem:13}
	Every primitive subset~$A$ of a finite ternary vector space~$V$ 
	is maximal sum-free and satisfies $|A|=\tfrac16(|V|+3|\Sym(A)|)$.
\end{lem}

This provides enough context for the main result of this article, 
which we would like to state next. 
 
\begin{thm}\label{thm:main}
	A subset $A$ of a finite ternary vector space $V$ 
	with $|A|>\frac16|V|$ is maximal sum-free if and only if 
	it is primitive. 
\end{thm}

\begin{cor}[Lev's periodicity conjecture]\label{cor:seva}
	Let $A\subseteq \FF_3^n$ be a maximal sum-free set. 
	If $A$ is aperiodic, then $n\ne 2$ and $|A|\le\frac12(3^{n-1}+1)$.
	\qed
\end{cor}

Let us finally point out that such problems can also be studied 
in $\FF_p^n$ for $p>3$ (see, 
e.g.,~\cites{VL23, LV23, RZ24a, RZ24b, RS71}
and the concluding remarks).

\section{Properties of primitive sets}\label{sec:pp}

In this section we collect some simple properties of primitive 
sets $A$, many of which are shown by unravelling their recursive
definition. In such situations it will be convenient to say that~$A$
is {\it derived} from $(H, U, W, X)$ if $H$, $U$, $W$, $X$ are as 
described in Definition~\ref{dfn:prim}\ref{it:prb}. Working towards 
Lemma~\ref{lem:13} we first look at symmetry sets.   

\begin{lem}\label{lem:21} 
	If a primitive set $A$ is derived from $(H, U, W, X)$, 
	then 
		\[	
		\Sym(A)=\Sym(X)\subseteq [U]\,.
	\]
	\end{lem}

\begin{proof}
	Suppose first that $v\in \Sym(A)$. Since $A\cap [H]=\vn$,
	we have $v\in [H]$. For this reason, $v$ can be expressed 
	as a difference of two vectors from $A\cap (-H)=X\cap (-U)$
	and, in particular, we have $v\in [U]$. This entails $W+v=W$
	and $v\in \Sym(X)$ follows. 
	
	Now suppose conversely that some vector $v\in \Sym(X)$ is given. 
	Owing to $X\subseteq U\cup (-U)$ we have $v\in [U]$, whence 
	$W=W+v$ and $v\in \Sym(A)$. 
\end{proof}

We proceed with an observation of Lev (cf.\ \cite{VL05}*{Lemma~5}).

\begin{lem}[Lev]\label{lem:22}
	Every subset $A$ of a ternary vector space $V$ 
	with $|A|>\frac13|V|$ satisfies $A-A=V$. 
\end{lem}

\begin{proof} 
	Let $v\in V$ be arbitrary. For cardinality reasons, the sets 
	$A$, $A+v$, and $A-v$ cannot be disjoint. Taking into account
	that $-v=2v$ this yields $v\in A-A$. 
\end{proof}

\begin{proof}[Proof of Lemma~\ref{lem:13}]
	When $A$ is a hyperplane we have $|V|=3|A|$ as well 
	as $\Sym(A)=[A]$, and everything is clear.
	We shall now establish the general case by induction on $n=\dim(V)$.
	In the induction step we may assume that $A$ is derived 
	from $(H, U, W, X)$. The induction hypothesis and Lemma~\ref{lem:21}
	yield 
		\[
		|A|
		=
		|W|+|X|
		=
		\tfrac12(|H|-|U|)+\tfrac16(|C(U)|+3|\Sym(A)|)
		=
		\tfrac16(|V|+3|\Sym(A)|)\,.
	\]
		
	Furthermore, $A$ is easily seen to be sum-free and it remains 
	to show that there is no vector $v\in V\sm A$ such that 
	$A\dcup\{v\}$ is still sum-free. Assume for the sake of 
	contradiction that some such vector $v$ exists. 
	The fact that $X$ is, by induction, maximal sum-free in $C(U)$ 
	implies $v\not\in C(U)$.
	
	So $v$ belongs to one of the five sets $W$, $-W$, $(-W)+(-U)$, 
	$W+U$, or $[H]$. The first two cases are impossible due to 
	$W\subseteq A$ and $-W\subseteq W+W\subseteq A+A$. 
	Next we observe that 
	Definition~\ref{dfn:prim}\ref{it:prb}\ref{it:priv} 
	yields some 
	vector $y\in (-U)\cap X$. Now $(-W)+(-U)\subseteq y-A\subseteq A-A$
	and $W+U\subseteq A-y\subseteq A-A$ rule out the third and fourth 
	possibility. Altogether this proves~${v\in [H]}$, which in view of
	Lemma~\ref{lem:22} and $v\not\in A-A$ yields 
	$|A\cap H|\le \frac13|H|$.
	
	But $|A\cap H|=|W|+|U\cap X|$ and unless $\dim(H/U)=1$ we already 
	have 
		\[
		|W|=\tfrac12(|H|-|U|)>\tfrac13|H|\,.
	\]
		This tells us $\dim(H/U)=1$ 
	and $U\cap X=\vn$. Now the second and third clause of 
	Definition~\ref{dfn:prim}\ref{it:prb} lead 
	to $X\subsetneqq (-U)$, which 
	contradicts the fact that $X$ is a maximal sum-free subset 
	of~$C(U)$.	
\end{proof}

Here is a rather surprising property of primitive sets, 
foreshadowed by~\cite{VL05}*{Lemma~4}, which will later 
be transferred to dense sum-free sets 
(see Proposition~\ref{prop:42}). 
 
\begin{lem}\label{lem:4A} 
	If $A$ is a primitive subset of a ternary vector space $V$, 
	then $4A=A+A+A+A$ does not contain the origin. 
\end{lem}

\begin{proof}
	If $A$ is a hyperplane with $0\not\in A$, 
	then $4A=A$ and the result is 
	clear. Arguing again by induction on $\dim(V)$ we suppose 
	that $A$ is derived from $(H, U, W, X)$ and consider the 
	canonical projection $\pi\colon V\lra V/[H]\cong \FF_3$ sending~$H$ 
	to~$1$. Assume for the sake of contradiction 
	that there are four vectors $x_1, x_2, x_3, x_4\in A$ such 
	that $x_1+x_2+x_3+x_4=0$. 
	
	The numbers $\pi(x_i)$ add up to zero as well, but due 
	to $A\cap [H]=\vn$ none of them is itself zero. 
	So without loss of generality we can 
	suppose $x_1, x_2\in A\cap H$ 
	and $x_3, x_4\in A\cap (-H)=X\cap (-U)$. In the quotient space 
	$V/U$ we thus have $x_1+x_2+U=U$ and, as $W$ is a $(H, U)$-half, 
	this is only possible if $x_1, x_2\in X\cap U$. Altogether
	this shows $0\in 4X$, which contradicts the induction hypothesis.
\end{proof}

\begin{lem}\label{lem:24}
	If $A$ denotes a primitive subset of a ternary vector space $V$
	which is not a hyperplane, then every 
	hyperplane $J\subseteq V$ satisfies  
	$|A|+|A\cap J|\le |J|$.
\end{lem}

\begin{proof}
	Let $A$ be derived from $(H, U, W, X)$.
	There are three special cases we would like to address before 
	completing the argument by induction on $\dim(V)$. 
	\begin{enumerate}[label=\nlabel]
		\item If $\dim(H/U)\ge 2$ and $J\ne H$, then the claim 
			follows from 
						\[
				|A|=|W|+|X|\le \tfrac12(|H|-|U|)+|U|\le \tfrac59|H|
			\]
						and $|A\cap J|\le |H\cap J|+|X|\le \frac49 |H|$.		
		\item Suppose next that $X$ is a hyperplane in $C(U)$.
			By Definition~\ref{dfn:prim}\ref{it:prb} we have $X=-U$ 
			and $\dim(H/U)\ge 2$. So we can assume $J=H$, as otherwise
			the previous case applies. Now we have indeed 
						\[
				|A|+|A\cap H|=2|W|+|U|=|H|\,.
			\]
			 		\item Suppose finally that $C(U)\subseteq J$. 
			This yields $J\ne H$ and in view of the first case we 
			can assume $\dim(H/U)=1$. Now $C(U)$ is a hyperplane and 
		 	we actually have $J=C(U)$ as well as 
			$|A|+|A\cap J|=|U|+2|X|\le 3|U|=|J|$.
	\end{enumerate}
	
	We are now ready for the induction step, where we assume that 
	none of the above cases applies. As $C(U)\cap J$ is either a 
	hyperplane in $C(U)$ or empty and $X$ fails to be a hyperplane,
	the induction hypothesis tells us $|X|+|X\cap J|\le |U|$,
	whence 
		\[
		|A|+|A\cap J|\le 2|W|+|U|=|H|\,. \qedhere
	\]
	\end{proof}

\begin{fact}\label{f:25}
	If $U\subseteq H$ are two affine ternary spaces 
	with $\dim(H/U)\ge 2$, then every $(H, U)$-half 
	contains an affine space of dimension $\dim(U)+1$. 
\end{fact}

\begin{proof}
	The claim reduces to its special 
	case $H=\FF_3^2$ and $U=\{(0, 0)\}$, which is easily 
	handled by means of a quick case analysis.
\end{proof}

\begin{lem}\label{lem:26}
	Every primitive set $A$ that fails to be a hyperplane 
	contains an affine space whose dimension exceeds that 
	of~$\Sym(A)$.
\end{lem}

\begin{proof}
	Let $A$ be derived from $(H, U, W, X)$. As long as $X$ 
	fails to be a hyperplane Lemma~\ref{lem:21} allows us 
	to appeal to an induction on the dimension of the ambient 
	ternary space. On the other hand, if $X$ is a hyperplane, 
	we can conclude $X=-U$ and $\dim(H/U)\ge 2$ 
	from Definition~\ref{dfn:prim}\ref{it:prb}. 
	But now Fact~\ref{f:25} yields the desired affine space 
	within~$W$. 
\end{proof}

The proof of Theorem~\ref{thm:main} will proceed by 
induction on the dimension of the ambient space~$V$. Since 
the intersection of a maximal sum-free subset of~$V$ with 
a proper subspace~$J$ of~$V$ is in general only a sum-free 
subset of~$J$ but not necessarily a maximal one
(even if this intersection is very dense), we need to study 
subsets of primitive sets as well. For brevity we shall call 
them {\it subprimitive} in the sequel. Utilising the previous 
lemma one can show that sufficiently large subprimitive sets 
have density at least $\frac56$ in certain affine spaces. 

\begin{lem}\label{lem:kaff}
	Let $B$ be a subprimitive subset of a ternary vector space $V$
	which is not contained in a hyperplane.
	If $|B|>\frac16(|V|+3^{k-1})$ holds for some positive integer $k$, 
	then there exists a $k$-dimensional affine subspace~$E$ 
	of~$V$ such that $|B\cap E|\ge \frac16(5|E|+3)$.
\end{lem}

\begin{proof}
	Let $A$ be a primitive superset of $B$. In view of 
		\[
		\tfrac16(|V|+3|\Sym(A)|)=|A|\ge |B|>\tfrac16(|V|+3^{k-1})
	\]
		the dimension $\rho$ of $\Sym(A)$ is at least $k-1$. 
	Since $A$ cannot be a hyperplane, Lemma~\ref{lem:26} yields 
	a $(\rho+1)$-dimensional affine space $Q\subseteq A$. Because of 
		\[
		|A\sm B|
		<
		\tfrac16(|V|+|Q|)-\tfrac16|V|
		=
		\tfrac16|Q|
	\]
		we have 
		\[
		|B\cap Q|\ge |Q|-|A\sm B|>\tfrac56 |Q|\,.
	\]
		
	As $Q$ can be expressed as a disjoint union of $k$-dimensional 
	affine spaces, this leads to some such space $E$ satisfying 
	$|B\cap E|> \frac56|E|$. Due to $|E|\equiv 3\pmod{6}$
	we actually have $|B\cap E|\ge\frac16(5|E|+3)$. 
\end{proof}

\begin{lem}\label{lem:28} 
	Let $A$ be a primitive subset of a ternary vector space $V$,
	but not a hyperplane. If 
	the subset $B\subseteq A$ satisfies $|B|>\frac16|V|$
	and the hyperplane $J$ is disjoint to $B$, then $J\cap A=\vn$
	holds as well.
\end{lem}

\begin{proof}
	Setting $T=\Sym(A)$ we have
		\[
		|A\sm B|
		<
		\tfrac16(|V|+3|T|)-\tfrac16|V|
		=
		\tfrac12|T|\,,
	\]
		for which reason every translate $T'$ of $T$ contained in $A$
	satisfies $|T\cap B|>\frac12|T|$. This takes immediate care 
	of the case $[T]\subseteq [J]$, where $J$ is a union of translates
	of $T$. On the other hand, if $[T]\not\subseteq [J]$, then
	$|T'\cap J|=\frac13|T|$ holds for every translate~$T'$ of~$T$. 
	So, if $A$ consists of at least two such translates, we reach the 
	contradiction $|A\sm B|\ge |A\cap J|\ge \frac23|T|$. But, as $A$ 
	fails to be a hyperplane, it cannot be a single translate of $T$ 
	either. 
\end{proof}

\section{Inductive arguments}

Let us recapitulate a fundamental result of additive combinatorics 
due to Kneser~\cites{Kn53, Kn55}.

\begin{thm}[Kneser]\label{thm:Kn}
	If $A$, $B$ are two finite nonempty subsets of an abelian 
	group~$G$ and $K=\Sym(A+B)$, then 
		\[
		|A+B|\ge |A+K|+|B+K|-|K|\,.
	\]
	\end{thm}

If $G$ is finite, then $|G|$ and the right side are divisible 
by $|K|$; this immediately implies the 
following well-known fact.    

\begin{cor}\label{cor:B}
	If two subsets $A$, $B$ of a finite abelian group $G$ satisfy 
	$|A|+|B|>|G|$, then $A+B=G$. \qed
\end{cor}

Our other applications of Kneser's theorem factorise through the 
following statement. 

\begin{lem}\label{lem:Kn}
	Let $A$, $B$, $C$ be three subsets of a finite abelian group~$G$
	such that $A+B$ is disjoint to $C$. If $2|A|+2|B|+|C|>2|G|$
	and $C\ne\vn$, then there is a subgroup $K$ of $G$ such that 
	$|A+K|+|B+K|=|G|$ and $C$ is contained in a single coset of $K$. 
\end{lem}

\begin{proof}
	Notice that $A=B=\vn$ would yield the contradiction $|C|>2|G|$. 
	If exactly one of $A$ and $B$ is empty, we can simply take $K=G$. 		So we may assume from now on that $A$ and $B$ are nonempty. 
	It will eventually turn out that $K=\Sym(A+B)$ is as required. 
	Since $A+B$ is a union of cosets of $K$, it is disjoint to $C+K$
	and, therefore, Kneser's theorem yields 
		\begin{equation}\label{eq:3106}
		|G|
		\ge 
		|A+B|+|C+K|
		\ge 
		|A+K|+|B+K|+|C+K|-|K|\,,	
	\end{equation}
		whence 
		\begin{equation}\label{eq:3109}
		|G|
		\ge
		|A+K|+|B+K|\,.
	\end{equation}
		Furthermore,~\eqref{eq:3106} entails
		\begin{align*}
		2|A+K|+2|B+K|+|C+K|
		&\ge
		2|A|+2|B|+|C|
		>
		2|G| \\
		&\ge
		|G|+|A+K|+|B+K|+|C+K|-|K|\,,
	\end{align*}
		i.e., $|A+K|+|B+K|>|G|-|K|$. As both sides are 
	multiples of $|K|$, this shows in combination with~\eqref{eq:3109}
	that $|A+K|+|B+K|=|G|$. In view of~\eqref{eq:3106} we can 
	conclude $|C+K|\le |K|$ and, consequently, $C$ is indeed 
	contained in a single coset of $K$. 	
\end{proof} 

The foregoing lemma will be used to handle a major case 
in our inductive proof of Theorem~\ref{thm:main}. We will
denote the statement that a given ternary vector space $V$ satisfies 
the induction hypothesis by $\Phi(V)$. More precisely, $\Phi(V)$ means 
that if $V'$ refers to a ternary vector space with $\dim(V')<\dim(V)$,
then every sum-free set $A'\subseteq V'$ with $|A'|>\frac16|V'|$ 
is subprimitive. Since Theorem~\ref{thm:main} is clear 
when $\dim(V)\le 2$, we will know that $\Phi(\FF_3^3)$ is true 
when looking at the three-dimensional case later. 

\begin{prop}\label{prop:33}
	Given a finite ternary vector space $V$ satisfying $\Phi(V)$, 
	let $A$ be a sum-free subset and $H$ a hyperplane not containing
	the origin. 
	If $|A|>\frac16|V|$, $[H]\cap A=\vn$, 
	and $\aff(A\cap(-H))$ is a proper affine subspace of $-H$, 
	then~$A$ is subprimitive. 
\end{prop}

\begin{proof}
	Due to the last assumption $U=\aff(H\cap (-A))$ is a proper 
	affine subspace of $H$. If~$U$ should be empty, 
	we have $A\subseteq H$ and the claim is clear. So we can 
	henceforth suppose~${U\ne\vn}$, whence $A\cap (-U)\ne\vn$. 
	We shall show that there exist 
	an $(H, U)$-half $W$ and a primitive subset $X$ of $C(U)$ 
	such that $A$ is contained in $W\dcup X$ ---a set which is 
	either primitive or contained in a hyperplane. 
	
	To this end we set $A_y=A\cap (U+y)$ for every $y\in [H]\sm [U]$.
	Since $A_y+A_{-y}\subseteq -U$ cannot hold with equality, 
	Corollary~\ref{cor:B} informs us that 
		\begin{equation}\label{eq:0042}
		|A_y|+|A_{-y}|\le |U|\,.
	\end{equation}
		By summing this over an appropriate half-system of 
	representatives of $[H]/[U]$ we infer 
	$|A\cap (H\sm U)|\le\frac12(|H|-|U|)$, which in turn leads to 
		\[
		|A\cap C(U)|
		>
		\tfrac16|V|-\tfrac12(|H|-|U|)
	 	=
		\tfrac16|C(U)|\,.
	\]
		
	Thus the inductive hypothesis $\Phi(V)$ yields a primitive 
	subset $X\subseteq C(U)$ covering $A\cap C(U)$.
	If $X$ is a hyperplane, then it needs to be equal to $-U$ 
	and 
		\begin{equation}\label{eq:3702}
		X\cap [H]=\vn
	\end{equation}
		follows. On the other hand, if $X$ is not a hyperplane, then 
	Lemma~\ref{lem:28} applied to $A\cap C(U)$ shows 
	that~\eqref{eq:3702} is still true. 
	
	\begin{clm}
		For every $z\in [H]\sm [U]$ one of the sets $A_z$, $A_{-z}$
		is empty.
	\end{clm}
	
	\begin{proof}
		It follows from~\eqref{eq:0042} that
				\[
			|A\cap (H\sm U)|
			\le 
			\tfrac12(|H|-3|U|)+|A_z|+|A_{-z}|\,.
		\]
				In view of $A\subseteq (H\sm U)\dcup C(U)$ and $|A|>\frac16|V|$
		this yields  
				\begin{equation}\label{eq:3131}
			\tfrac32|U|
			<
			|A_z|+|A_{-z}|+|A\cap C(U)|\,.
		\end{equation}
				Next we contend that
				\begin{equation}\label{eq:3132}
			|A\cap C(U)|+|A\cap U|\le |U|\,.
		\end{equation}
				If $A\cap C(U)$ fails to be a hyperplane, this can be seen
		by applying Lemma~\ref{lem:24} to $C(U)$, $X$ here in place 
		of $V$, $A$ there. In the remaining case we have $X=-U$ 
		and~\eqref{eq:3132} is obvious. 
		
		By doubling~\eqref{eq:3131} and taking~\eqref{eq:3132} 
		into account we learn 
				\[
			2|U|<2|A_z|+2|A_{-z}|+|A\cap (-U)|\,.
		\]
				Now Lemma~\ref{lem:Kn} shows that there is a subspace~$K$ 
		of~$[U]$ such that $|A_z+K|+|A_{-z}+K|=|U|$ and $A\cap (-U)$ 
		is contained in a translate of $K$. 
		Appealing to~$\aff((-U)\cap A)=-U$ we infer $K=[U]$
		and, therefore, one of $A_z$, $A_{-z}$ has to be empty. 
	\end{proof}
	
	As the claim demonstrates, there is an $(H, U)$-half $W$
	covering $A\cap (H\sm U)$. Clearly, if 
	Definition~\ref{dfn:prim}\ref{it:prb}\ref{it:priii} is 
	satisfied, then~$A$
	is contained in the primitive set derived from $(H, U, W, X)$. 
	It remains to treat the exceptional case that $\dim(H/U)=1$ 
	and $X=-U$. Now 
		\[
		W\dcup(-U)\dcup (U+(-W))
	\]
		is a hyperplane in $V$ covering $A$, so that $A$ is again 
	subprimitive. 
\end{proof}

The three-dimensional case of Theorem~\ref{thm:main} was first proved  
by Lev (see~\cite{VL05}*{Lemma~3}) using an exhaustive case analysis.
For the sake of completeness we would like to mention an 
alternative argument based on Proposition~\ref{prop:33}.
   
\begin{cor}[Lev]\label{cor:3}
	Every sum-free set $A\subseteq \FF_3^3$ with $|A|\ge 5$ 
	is subprimitive and contains a line. 
\end{cor}

\begin{proof}
	We shall show first that every five-element sum-free set 
	$A\subseteq \FF_3$ contains four coplanar points. 
	To this end we observe that thirteen lines of~$\FF_3^3$ pass 
	through the origin. Five of them intersect~$A$, while the 
	eight other ones do not. There are $\binom52>8$ pairs of points 
	from~$A$ and each pair determines the direction of a line. Thus 
	there are either two parallel lines intersecting $A$ at least 
	twice, or there is a line intersecting $A$ thrice. In both cases,  
	there are indeed four coplanar points in $A$.  

	Now let $A\subseteq \FF_3^3$ be an arbitrary sum-free set with 
	$|A|\ge 5$. If the entire set is coplanar, then it is clearly 
	subprimitive and a quick case analysis discloses that $A$ contains
	three collinear points. 
	Otherwise, there is a subset $A'\subseteq A$ of size five 
	which fails to be coplanar. By the first paragraph there is 
	a plane $H$ with $|A'\cap H|=4$, Proposition~\ref{prop:33}
	shows that $A'$ is primitive, and Lemma~\ref{lem:13} 
	yields $A'=A$. Finally, by Lemma~\ref{lem:26} $A$ contains a line.
\end{proof}

Here is a straightforward consequence, whose easy 
proof we omit. 

\begin{cor}\label{cor:stuss}
	If $A\subseteq \FF_3^3$ is a subprimitive set of size four, 
	then either $A$ is contained in a plane or one point in $A$ 
	is the sum of the three other ones. \qed
\end{cor}

A simple averaging arguments transfers the last part of 
Corollary~\ref{cor:3} to higher dimensions. 

\begin{cor}\label{cor:k1}
	For every at least three-dimensional ternary vector space $V$, 
	every sum-free set $A\subseteq V$ with $|A|>\frac16|V|$ contains 
	a line. 
\end{cor}

\begin{proof}
	By averaging there exists a three-dimensional subspace $L\le V$
	such that 
		\[
		|A\cap L|
		\ge
		\frac{26|A|}{|V|-1}
		>
		\frac{26}6
		>
		4\,,
	\]
		i.e., $|A\cap L|\ge 5$. So by Corollary~\ref{cor:3} $A\cap L$ 
	contains a line.
\end{proof}

On first reading the next result might look like a variant of 
Proposition~\ref{prop:33}. However, as it is quite unrelated to 
the recursive structure of primitive sets, it lies much closer 
to the surface. 

\begin{prop}\label{prop:37}
	Let~$A$ be a sum-free subset of a finite ternary vector space~$V$
	of size $|A|>\frac16|V|$.
	If there exists a hyperplane $H$ not passing through the origin 
	such that $A\cap H=\vn$ and $\aff(A\cap [H])\ne [H]$, then $A$ 
	is subprimitive. 
\end{prop}

\begin{proof}
	Let $U$ be a hyperplane in $[H]$ covering 
	$\aff(A\cap [H])$.
	If $|A\cap (-H)|>\frac13|H|$, then Lemma~\ref{lem:22}
	yields $A\cap [H]=\vn$, whence $A$ is covered by the 
	hyperplane $-H$. 
	Suppose now that $|A\cap (-H)|\le \frac13|H|$, wherefore 
		\begin{equation}\label{eq:3424}
		|A\cap U|
		> 
		\tfrac16|H|
		=
		\tfrac12|U|\,.
	\end{equation}
	In particular, $U$ does not contain the origin and there exists 
	a partition of $-H$ into three translates $X$, $Y$, $Z$ of $U$ such 
	that $Y=X+U$ and $Z=Y+U$. For reasons of cyclic symmetry 
	we can suppose 
		\begin{equation}\label{eq:3420}
		|A\cap Z|=\min\{|A\cap X|, |A\cap Y|, |A\cap Z|\}\,.
	\end{equation}
		If $|A\cap X|+|A\cap U|>|U|$, then $A$ is disjoint to $X\pm U$,
	whence $A$ is covered by the hyperplane containing $X$ and $U$. 
	So we may assume from now on that 
		\begin{equation}\label{eq:3442}
		|A\cap Y|
		\overset{\eqref{eq:3420}}{\ge}
		\tfrac12(|A\cap Y|+|A\cap Z|)
		\ge
		\tfrac12(|A|-|U|)
		>
		\tfrac14|U|
	\end{equation}
		and, similarly, that $|A\cap X|>\frac14|U|$. 
	In particular, neither $A\cap X$ nor $A\cap Y$ is empty. 
	Together with 
		\begin{align*}
		2|A\cap X|+2|A\cap U|+|A\cap Y|
		&=
		|A|+|A\cap U|+(|A\cap X|-|A\cap Z|) \\
		&\overset{\eqref{eq:3420}}{\ge}
		|A|+|A\cap U| 
		\overset{\eqref{eq:3424}}{>}
		\tfrac32|U|+\tfrac12|U|
		=
		2|U|
	\end{align*}
	and Lemma~\ref{lem:Kn} this tells us that there is a subspace $K$
	of $[U]$ such that
		\[
		|(A\cap X)+K|+|(A\cap U)+K|=|U|
	\]
		and $A\cap Y$ is covered by a single translate of $K$. 
	We cannot have $K=[U]$, for then one of $A\cap X$ or $A\cap U$
	had to be empty. So~\eqref{eq:3442} reveals $\dim(U/K)=1$
	and we see that $A\cap X$ is contained in a single translate 
	of $K$, while $A\cap U$ is covered by two translates of $K$. 
	
	Let $\psi\colon V\lra \FF_3^3$ be the epimorphism with 
	kernel~$K$ satisfying 
		\[
		\psi[A\cap U]=\{(0, 1, 1), (0, 1, 2)\} 
		\quad \text{ and } \quad 
		\psi[A\cap X]=\{(1, 0, 0)\}\,.
	\]
		Since $|A\cap Z|\le |A\cap Y|\le |K|$, we have 
		\[
		|A\cap X|+|A\cap U|>|A|-2|K|>2|K|
	\]
		and, hence, 
		\[
		|A\cap \psi^{-1}(1, 0, 0)|+|A\cap \psi^{-1}(0, 1, 1)|>|K|\,.
	\]
		By Corollary~\ref{cor:B} this yields 
	$(1, 0, 0)\pm (0, 1, 1)\not\in \psi[A]$
	and, similarly, $(1, 0, 0)\pm (0, 1, 2)\not\in \psi[A]$
	holds as well. This leaves us with $\psi[A\cap Y]=\{(1, 1, 0)\}$
	and $\psi[A\cap Z]=\{(1, 2, 0)\}$, so that altogether 
		\[
		\psi[A]
		=
		\{(0, 1, 1), (0, 1, 2), (1, 0, 0), (1, 1, 0), (1, 2, 0)\}\,.
	\]
		As this subset of $\FF_3^3$ is primitive, $A$ is indeed 
	subprimitive. 
\end{proof}

Both propositions of this section combine as follows. 

\begin{conclusion}\label{conc:310}
	Assume $\Phi(\FF_3^n)$, where $n\ge 2$, and 
	set $V_{ij}=\{(i, j)\}\times \FF_3^{n-2}$ for every 
	point $(i, j)\in\FF_3^2$.  
	Let $A\subseteq \FF_3^n$ be a sum-free set with
	${|A|>\tfrac12\cdot 3^{n-1}}$ 
	and $|A\cap V_{01}|> \tfrac12\cdot 3^{n-2}$.  
	If for every $i\in \FF_3$ at least one of the sets $A\cap V_{1i}$,
	$A\cap V_{2, 1-i}$ is empty, then $A$ is subprimitive. 
\end{conclusion}

\begin{proof}
	Set $A_{ij}=A\cap V_{ij}$ for every $(i, j)\in\FF_3^2$.
	Corollary~\ref{cor:B} yields  $A_{01}-A_{01}=V_{00}$
	and $A_{01}+A_{01}=V_{02}$, whence $A_{00}=A_{02}=\vn$.
	By symmetry we can assume further that 
		\begin{enumerate}[label=\nlabel]	
		\item\label{it:51} either $A_{10}=A_{11}=A_{12}=\vn$
		\item\label{it:52} or $A_{10}=A_{12}=A_{20}=\vn$.
	\end{enumerate}
		
	If~\ref{it:51} holds, then the hyperplane 
	$H=V_{10}\dcup V_{11}\dcup V_{12}$ fulfils the hypothesis 
	of Proposition~\ref{prop:37} and $A$ is indeed subprimitive.
	Similarly, in case~\ref{it:52} we can apply 
	Proposition~\ref{prop:33} to the hyperplane 
	$H=V_{01}\dcup V_{11}\dcup V_{21}$.
\end{proof}

\section{Quadruple sums}

Summarising our current state of knowledge, we have shown 
the `easy' direction of Theorem~\ref{thm:main} for all ternary 
spaces $V$, and we know that the `difficult' direction holds 
in the low-dimensional cases $\dim(V)\le 3$ 
(cf. Corollary~\ref{cor:3}), which gives
the induction hypothesis~$\Phi(\FF_3^4)$. Before we can complete 
the general induction step, however, we need to study the 
case $n=4$ separately (cf.\ Lemma~\ref{lem:42}). Roughly speaking, 
this is because that case is required for establishing that dense 
sum-free sets $A$ satisfy $0\not\in 4A$. 

We would like to point out that, as reported in~\cite{VL05}, 
Lev verified the case $n=4$ of Corollary~\ref{cor:seva} 
electronically. It seems reasonable to believe that the programme 
he wrote could confirm Lemma~\ref{lem:42} as well. Nevertheless, 
we did not pursue this direction here. Rather, we argue that 
the results from the previous section allow us to split this 
problem into a manageable number of cases. We commence with 
sets containing two parallel lines.    

\begin{lem}\label{lem:41}
	If a sum-free set $A\subseteq \FF_3^4$ with $|A|\ge 14$ 
	contains two parallel lines, then it is subprimitive.
\end{lem}

\begin{proof}
	Arguing indirectly we assume that $A$ is not subprimitive. 
	For every $k\in\{1, 2, 3, 4\}$ let $\pi_k\colon\FF_3^4\lra\FF_3$
	be the projection onto the $k^{\mathrm{th}}$ coordinate.  
	Set $V_{ij}=\{(i, j)\}\times \FF_3^2$ for 
	every point $(i, j)\in \FF_3^2$. Pick a two-dimensional affine 
	space $Q\subseteq \FF_3^4$ such that $A\cap Q$ contains two 
	parallel lines and, subject to this, $|A\cap Q|$ is maximal. 
	Due to $|A\cap Q|\ge 6>3$ the origin cannot be in $Q$. 
	So an appropriate choice of coordinates allows us to 
	assume $Q=V_{01}$ and that 
	$B_3=\{(0, 1, y, z)\in \FF_3^4\colon y\ne 0\}$ is a subset of $A$.
	
	\begin{clm}\label{clm:42}
		If two points $p, q\in A$ satisfy $\pi_1(p)=\pi_1(q)\ne 0$,
		then $\pi_3(p)=\pi_3(q)$.
	\end{clm}
	
	\begin{proof}
		Assume for the sake of contradiction that this fails, i.e.,
		that there are $p, q\in A$ such that $\pi_1(p)=\pi_1(q)\ne 0$,
		but $\pi_3(p)\ne \pi_3(q)$. 
		Owing to $\pm(p-q)\not\in B_3$ we have $\pi_2(p)=\pi_2(q)$.
		Without loss generality we can assume 
		that $p, q\in V_{11}$. We intend to show that 
		$H=\pi_2^{-1}(1)$ satisfies the assumptions 
		of Proposition~\ref{prop:33}. Due to $B_3-B_3=V_{00}$, 
		$\{p, q\}-B_3=V_{10}$, and $B_3-\{p, q\}=V_{20}$ we have 
		indeed $A\cap [H]=\vn$. Moreover $B_3+B_3=V_{02}$ 
		and $B_3+\{p, q\}=V_{12}$ show that $A\cap (-H)$ is contained 
		in the affine space $V_{22}$. 
	\end{proof}
	
	By Conclusion~\ref{conc:310} there exists some $i\in \FF_3$
	for which neither $A\cap V_{1i}$ nor $A\cap V_{2, 1-i}$ is empty.
	Pick $x\in A\cap V_{1i}$, $y\in A\cap V_{2, 1-i}$
	and set $r_3=\pi_3(x)$, $s_3=\pi_3(y)$.
	As $x+y\not\in B_3$, 
	we have $r_3+s_3=0$. By Claim~\ref{clm:42} all points $x'\in A$ 
	with $\pi_1(x')=1$ satisfy $\pi_3(x')=r_3$ and, similarly, all 
	$y'\in A$ with $\pi_1(y')=2$ have the property 
	$\pi_3(y')=s_3=2r_4$. Consequently, $A\sm B_3$ is contained in 
	the three-dimensional vector 
	space $W_3=\{z\in\FF_3^4\colon \pi_3(z)=r_3\pi_1(z)\}$. 
	
	In view of $|A\sm B_3|=8>5$ it follows  
	that some two-dimensional affine space covers $A\sm B_3$. 
	By our extremal choice 
	of $Q=V_{01}$ this yields $|A\cap V_{01}|\ge 8$ and we can 
	assume that the points $(0, 1, 0, \pm 1)$
	are in $A$, which causes 
	$B_4=\{(0, 1, y, z)\in \FF_3^4\colon z\ne 0\}$
	to be a subset of $A$. By exchanging the r\^{o}les of the 
	third and fourth coordinate in the arguments presented to far we
	see that $A\sm B_4$ lives in a three-dimensional space of  
	the form $W_4=\{z\in\FF_3^4\colon \pi_4(z)=r_4\pi_1(z)\}$. 
	But now $\dim(W_3\cap W_4)=2$ yields $|A\cap W_3\cap W_4|\le 3$
	and we are led to the contradiction 
	$|A|=|B_3\cup B_4|+|A\cap W_3\cap W_4|\le 11$. 
\end{proof}

\begin{lem}\label{lem:42}
	Every sum-free set $A\subseteq \FF_3^4$ with $|A|\ge 14$ 
	is subprimitive.
\end{lem}

\begin{proof}
	Let $\psi\colon \FF_3^4\lra\FF_3^3$ be the projection deleting 
	the last coordinate. Since $A$ contains a line 
	(by Corollary~\ref{cor:k1}), we may assume that there exists a	
	point $p\in \FF_3^3$ such that $\psi^{-1}(p)\subseteq A$.
	Set $B=\{q\in \FF_3^3\colon |\psi^{-1}(q)\cap A|\ge 2\}$ 
	and $C=\{q\in \FF_3^3\colon \psi^{-1}(q)\cap A\ne\vn\}$.
	As~$A$ is sum-free, we have 
		\begin{equation}\label{eq:4139}
		(B+B)\cap C=(B-B)\cap C=\vn\,.
	\end{equation}
		In particular, $B$ is sum-free.
	If there exists a point $q\ne p$ in $\FF_3^3$ 
	with $\psi^{-1}(q)\subseteq A$, then $A$ contains two parallel 
	lines and Lemma~\ref{lem:41} tells us that $A$ is indeed 
	subprimitive. So we may assume from now on that $p$ is the only 
	such point. In view of $|A|\ge 14$ this implies 
		\begin{equation}\label{eq:4146}
		|B|+|C|\ge 13\,.
	\end{equation}
		Moreover, $C$, $C+p$, and $C-p$ are disjoint subsets of $\FF_3^3$,
	whence 
		\begin{equation}\label{eq:4151}
		|C|\le 9\,.
	\end{equation}
		
	\begin{clm}
		There are four coplanar points in $B$.
	\end{clm}
	
	\begin{proof}
		If $|B|\ge 5$ this follows from the already established 
		case $V=\FF_3^3$ of Theorem~\ref{thm:main}. 
		Due to~\eqref{eq:4146} and~\eqref{eq:4151} it remains to 
		discuss the case that $|B|=4$ and $|C|=9$. Take an arbitrary 
		point $c\in C\sm B$. Using~\eqref{eq:4139} it is easy to 
		show that $B\cup\{c\}$ is sum-free. Hence, $B$ is subprimitive
		and by Corollary~\ref{cor:stuss} either we are done or one 
		point of $B$ is the sum of the three other ones. For reasons 
		of symmetry we may assume that
				\[
			B=\{(1, 0, 0), (0, 1, 0), (0, 0, 1), (1, 1, 1)\}\,.
		\]
				But now~\eqref{eq:4139} entails 
				\[
			C\subseteq B\cup\{(1, 2, 2), (2, 1, 2), (2, 2, 1)\}\,,
		\]
				which contradicts $|C|=9$.
	\end{proof}
	
	Let $H\subseteq \FF_3^3$ be a plane intersecting $B$ in at 
	least four points. This plane cannot contain the origin, 
	Lemma~\ref{lem:22} yields $B-B=[H]$, and Kneser's theorem 
	implies $|(B+B)\cap (-H)|\ge 7$. Hence,~\eqref{eq:4139} tells 
	us $C\cap [H]=\vn$ and that $C\cap (-H)$ can be covered by 
	a line. So by Proposition~\ref{prop:33} applied to $\psi^{-1}[H]$	
	the set $A$ is indeed subprimitive.
\end{proof}

Now we keep a promise given in Section~\ref{sec:pp}.

\begin{prop}\label{prop:42}
	If a sum-free subset $A$ of a ternary space~$V$ satisfies 
	$|A|>\frac16|V|$, then~${0\not\in 4A}$.
\end{prop}

\begin{proof}
	Assume for the sake of contradiction that $x_1$, $x_2$, $x_3$,
	and $x_4$ are four elements of $A$ summing up to $0$, and let $L$
	be the subspace of $V$ generated by these vectors. It is easily 
	seen that~$L$ is three-dimensional. 
	Moreover, $\{x_1, x_2, x_3, x_4\}$ is a maximal sum-free subset 
	of~$L$, whence $|A\cap L|=4$.
	
	Let us now consider any four-dimensional subspace $Q$ of $V$
	that contains $L$. 
	By Lemma~\ref{lem:4A} $A\cap Q$ cannot be subprimitive in $Q$ 
	and, therefore, Lemma~\ref{lem:42}
	entails $|A\cap Q|\le 13$. This proves $|A\cap (Q\sm L)|\le 9$
	and by averaging over $Q$ we infer 
		\[
		|A|
		\le 
		|A\cap L|+\tfrac{9}{54}|V\sm L|
		=
		\tfrac16(|L|-3)+\tfrac16(|V|-|L|)
		<
		\tfrac16|V|\,,
	\]
		which is absurd. 
\end{proof}

\section{Proof of the main result}

It turns out that Conclusion~\ref{conc:310} allows us to 
prove the main theorem by induction, because there always 
exists a coordinate system with respect to which its hypotheses
are valid. For the first of them this is shown as follows. 

\begin{prop}\label{prop:51}
	Let $V$ be an at least three-dimensional ternary vector 
	space satisfying~$\Phi(V)$. 
	For every sum-free set $A\subseteq V$ with $|A|>\frac16|V|$
	there is an affine subspace $Q$ of $V$ such that 
	$\dim(V/Q)=2$ and $|A\cap Q|\ge\frac12(|Q|+3)$.
\end{prop}

\begin{proof}
	Recall that $A$ contains a line by Corollary~\ref{cor:k1}. 
	In other words, there is a one-dimensional affine subspace~$L$
	such that $|A\cap L|=\frac16(5|L|+3)$.  
	Now let $k\le \dim(V)-2$ be the largest integer such that 
	$|A\cap U|\ge\frac16(5|U|+3)$ holds for some $k$-dimensional 
	affine space $U$. If $k=\dim(V)-2$ we are done, so suppose  
	from now on that this is not the case. 
	Note that $|A\cap U|>\frac12|U|$ yields $0\not\in U$ 
	and $|C(U)|=3|U|$.
	Setting $q=\frac{|V|}{9|U|}$  
	an averaging argument shows that there is a hyperplane $J$
	such that $C(U)\le J\le V$ and 
		\begin{align*}
		|A\cap J|
		&\ge 
		|A\cap C(U)|+\frac{|J|-|C(U)|}{|V|-|C(U)|}|A\sm C(U)| 
	 	=
		\frac{q-1}{3q-1}|A|+\frac{2q}{3q-1}|A\cap C(U)| \\
		&>
		\frac{|J|}{6}
		+
		\frac{5q|U|-|J|}{9q-3}
		=
		\frac{|J|}{6}+\frac{2q|U|}{9q-3}
		>
		\frac{3|J|+4|U|}{18}\,.
	\end{align*}
		
	In particular, $A\cap J$ is subprimitive in $J$ by our 
	assumption $\Phi(V)$.  
	Furthermore, if $A\cap J$ is not contained in a hyperplane of~$J$, 
	then Lemma~\ref{lem:kaff} yields a $(k+1)$-dimensional affine 
	subspace~$U'$ of~$J$ such that $|A\cap U'|\ge \frac16(5|U'|+3)$,
	contrary to the maximality of $k$. Hence there is a 
	hyperplane $Q$ in $J$ covering $A\cap J$. Due to 
		\[
		|A\cap Q|
		=
		|A\cap J|
		>
		\tfrac16(|J|+4)
		>
		\tfrac12(|Q|+1)
	\]
		we have thereby found the desired affine subspace $Q$.  
\end{proof}

\begin{proof}[Proof of Theorem~\ref{thm:main}]
	We may assume $V=\FF_3^n$ for some $n\ge 2$ and, for reasons 
	of induction, that $\Phi(V)$ holds. Because of 
	Lemma~\ref{lem:13} it suffices to show that every sum-free 
	set $A\subseteq \FF_3^n$ of size $|A|>\frac16\cdot 3^n$ 
	is subprimitive.
	Let $\phi\colon \FF_3^n\lra \FF_3^2$ be the projection onto 
	the first two coordinates and set $V_{ij}=\phi^{-1}(i, j)$,
	$A_{ij}=A\cap V_{ij}$ for every point $(i, j)\in\FF_3^2$. 
	Owing to Proposition~\ref{prop:51} we may assume 
	that $|A_{01}|\ge \frac12(3^{n-2}+3)$, which yields 
		$A_{01}+A_{01}=V_{02}$.
	
	On the other hand, for every $i\in \FF_3$ 
	Proposition~\ref{prop:42} discloses 
	$A_{01}+A_{01}+A_{1i}+A_{2,1-i}\ne V_{00}$, 
	which together with 
	$(0, 1)+(0, 1)+(1, i)+(2, 1-i)=(0, 0)$ informs us that 
	at least one of the sets $A_{1i}$, $A_{2, 1-i}$ is empty. 
	In the light of Conclusion~\ref{conc:310} this shows that~$A$
	is indeed subprimitive.
\end{proof}

\section{Concluding remarks}
Given an arbitrary finite abelian group $G$ let $t(G)$ 
be the largest possible size of an aperiodic maximal 
sum-free set $A\subseteq G$.
For some groups $G$ no such sets exist and then we put $t(G)=0$.
It seems plausible that this invariant of $G$ tends to be `small' 
if $G$ has `many' subgroups and thus the special case that $G$ is 
a finite vector space might be especially interesting. 

As pointed out in the introduction, Davydov and Tombak~\cite{DT} 
proved 
\[
	t(\FF_2^n)
	=
	\begin{cases}
		1 & \text{ if $n=1$} \\
		0 & \text{ if $n=2, 3$} \\
		2^{n-2}+1 & \text{ if $n\ge 4$}\,,
	\end{cases}
\]
while Lev's periodicity conjecture asserts 
\[
	t(\FF_3^n)
	=
	\begin{cases}
		\tfrac12(3^{n-1}+1) & \text{ if $n\ne 2$} \\
		0 & \text{ if $n=2$}\,.	
	\end{cases}
\]
For primes $p$ satisfying $p\equiv 1\pmod{3}$ results 
of Rhemtulla and Street~\cites{RS70, RS71} imply 
\[
	t(\FF_p^n)=\tfrac13(p-1)p^{n-1}\,,
\]
while~\cite{RZ24a} yields 
\[
	t(\FF_p^n)=
	\begin{cases}
		\frac13(p+1) & \text{ if } n=1 \\
		\frac13(p-2)p^{n-1} & \text{ if } n\ge 2
	\end{cases}
\]
whenever $p\equiv 2\pmod{3}$ and $p\ge 11$. This leaves only the 
case $p=5$ open. It is not difficult to check 
directly that $t(\FF_5)=2$ and $t(\FF^2_5)=5$ (see~\cite{LV23}). 
Moreover, the main result of~\cite{RZ24b} entails $t(\FF_5^3)=28$. 
For $n\ge 4$, however, nothing seems to be known about $t(\FF_5^n)$ 
and nobody ever conjectured anything.

\end{document}